%% file: main.tex
\documentclass[11pt,a4paper]{article}
\usepackage[T1]{fontenc}
\usepackage[utf8]{inputenc}
\usepackage{lmodern}
\usepackage[margin=27mm]{geometry}
\usepackage{amsmath,amssymb,amsthm,mathtools,booktabs,array,longtable}
\usepackage{graphicx,microtype,enumitem}
\usepackage[colorlinks=true,linkcolor=blue!45!black,citecolor=blue!45!black,urlcolor=blue!45!black]{hyperref}
\usepackage{xcolor}
\hypersetup{pdftitle={Stationary Coexistence in Stochastic Gause Systems with General Trophic Responses},pdfauthor={Andres Sanchez; Leon A. Valencia; Jorge M. Ramirez Osorio}}
\setlist{nosep}
\newtheorem{theorem}{Theorem}[section]
\newtheorem{proposition}[theorem]{Proposition}
\newtheorem{lemma}[theorem]{Lemma}
\newtheorem{corollary}[theorem]{Corollary}
\theoremstyle{definition}
\newtheorem{assumption}[theorem]{Assumption}
\newtheorem{definition}[theorem]{Definition}
\theoremstyle{remark}
\newtheorem{remark}[theorem]{Remark}
\newcommand{\R}{\mathbb R}
\newcommand{\E}{\mathbb E}
\newcommand{\Prob}{\mathbb P}
\newcommand{\Lcal}{\mathcal L}
\newcommand{\mux}{\mu_X}
\newcommand{\pii}{\pi^{\circ}}
\newcommand{\dd}{\,\mathrm d}
\newcommand{\sig}{\sigma_*^2}
\newcommand{\lambdax}{\lambda_x^0}
\newcommand{\lambday}{\lambda_y}
\newcommand{\st}{\leq_{\mathrm{st}}}
\input{numerics/results.tex}
\input{numerics/table.tex}
\input{author_details.tex}
\title{Stationary Coexistence in Stochastic Gause Systems\\with General Trophic Responses}
\author{Andr\'es S\'anchez\AuthorOneDetails \and Leon A. Valencia\AuthorTwoDetails \and Jorge M. Ramirez Osorio\AuthorThreeDetails}
\date{}
\begin{document}
\maketitle
\begin{abstract}
We study stationary coexistence in stochastic Gause predator--prey systems with independent multiplicative environmental noise and trophic responses depending on both population densities. A weighted population balance cancels the trophic interaction and yields moments of order greater than one without imposing predator self-regulation. These estimates give a direct occupation-measure proof of a sufficient coexistence criterion based on prey establishment and predator invasion of the prey-only stationary environment. For responses whose rare-predator limit is strictly increasing and whose value at positive predator density never exceeds that limit, we also establish necessity of positive predator invasion whenever the prey-only stationary law exists. If the prey per-capita growth attains its global maximum at zero, interior stationarity itself forces prey establishment, giving a necessary-and-sufficient criterion beyond logistic growth. Logistic and bounded saturating growth laws provide explicit Gamma and beta-prime resident distributions. Applications include Holling, Ivlev, Beddington--DeAngelis, and Crowley--Martin responses. Responses with the same rare-predator limit share the coexistence threshold, although predator interference changes the interior distribution. Reproducible quadrature and simulations illustrate this distinction. The analysis builds on existing stochastic persistence theory and identifies precisely the additional comparison argument needed for necessity, including exclusion of interior invariant measures at zero predator invasion.
\end{abstract}


\section{Introduction}
An obligate predator can persist only if prey consumption compensates for mortality and environmental fluctuations. In a fluctuating prey environment, the relevant comparison is between these losses and feeding averaged against the resident prey distribution. The existence of a stationary distribution with both species present is therefore naturally related to the growth rate of a rare predator. Whether this invasion test is also necessary depends on how predation changes prey availability and how conspecific predators affect feeding. Classical prey-dependent feeding responses originate in Holling's analysis \cite{Holling}; interference mechanisms motivate the predator-dependent responses of Beddington \cite{Beddington}, DeAngelis, Goldstein, and O'Neill \cite{DeAngelis}, and Crowley and Martin \cite{Crowley}.

We consider the Gause system
\begin{equation}\label{eq:model}
\begin{aligned}
\mathrm dX_t&=[X_tF(X_t)-\Phi(X_t,Y_t)Y_t]\,\mathrm dt+\sigma_xX_t\,\mathrm dB_{x,t},\\
\mathrm dY_t&=Y_t[\beta\Phi(X_t,Y_t)-\delta]\,\mathrm dt+\sigma_yY_t\,\mathrm dB_{y,t},
\end{aligned}
\end{equation}
where $X$ and $Y$ are prey and predator densities, $F$ is the prey per-capita growth, and $\Phi(x,y)$ is the feeding rate of one predator. The Brownian motions are independent and $\beta,\delta,\sigma_x,\sigma_y>0$. Predators have negative growth in the absence of prey, and there is no additional term of predator self-regulation. Functional responses and their biological interpretation are discussed, for example, in \cite{Kot}.

The stochastic persistence approach relates coexistence to average logarithmic growth rates on the extinction boundary. Schreiber, Bena\"im, and Atchad\'e \cite{SBA} developed invasion criteria in compact settings. Hening and Nguyen \cite{HNkol} treated stochastic Kolmogorov systems on noncompact spaces, and Hening, Nguyen, and Chesson \cite{HNC} developed a broader coexistence and extinction theory. Hening and Li \cite{HL} studied stationary distributions and applications to ecological systems. These results separate boundary repulsion from control at infinity. Both are essential: positive invasion rates alone do not prevent occupation measures from escaping to large population densities.

For \eqref{eq:model}, the weighted population variable $W=1+x+y/\beta$ satisfies
\begin{equation}\label{eq:balance-intro}
\Lcal W=xF(x)-\delta y/\beta.
\end{equation}
The trophic terms cancel. Weighted estimates of this type have antecedents in stochastic food chains \cite{HNfood}; the linear-response logistic case is already covered there. Here the balance gives a common verification for general prey- and predator-dependent responses with $\Phi(x,y)/x$ bounded. We use the established invasion principle rather than claim a new general persistence mechanism. In particular, the sufficient coexistence conclusion is also recoverable from Theorem~4.4 and Corollary~4.5 of M. Bena\"im's unpublished manuscript \emph{Stochastic Persistence}, \href{https://arxiv.org/abs/1806.08450v3}{arXiv:1806.08450v3} (2023). Appendix~\ref{app:benaim} verifies this explicitly using the extended-generator assumptions in that version.

The converse uses additional ecological structure. Suppose
\begin{equation}\label{eq:no-fac-intro}
\Phi(x,y)\leq\Phi(x,0),\qquad x>0,\quad y\geq0,
\end{equation}
and $x\mapsto\Phi(x,0)$ is strictly increasing. Condition~\eqref{eq:no-fac-intro} excludes positive facilitation relative to the rare-predator regime; it does not require monotonicity at all positive predator densities. A stationary prey population subject to predation is stochastically dominated by the prey-only resident distribution. The two distributions must differ, since predation has strictly positive average intensity in an interior stationary regime. Strictly increasing feeding then gives a strict inequality between the corresponding average feeding rates, proving positive predator invasion. Comparison with the resident and strict inequalities of average growth have a close methodological antecedent in Nguyen and Strickler \cite{NS}, including a Rosenzweig--MacArthur application at criticality. Our result formulates the comparison for the structural response class in \eqref{eq:model}.

The comparison is not restricted to logistic growth. Conditional on prey establishment it gives necessity for the general dissipative $F$ considered below. If $F(x)\leq F(0)$, the stationary prey identity additionally forces prey establishment, and the criterion becomes necessary and sufficient. This covers logistic growth and bounded, nonlogistic growth laws. A bounded strictly increasing response causes no difficulty: strict monotonicity, rather than unboundedness, separates two distinct ordered laws.

Related response-specific results require careful comparison. Du, Nguyen, and Yin \cite{DNY} study Beddington--DeAngelis dynamics with positive predator self-regulation. Hening and Li \cite{HL}, Section~7.2, state an invasion criterion for Beddington--DeAngelis responses and include examples without predator self-regulation. Li and Guo \cite{LG} consider general functional responses with additional regulation and perturbations; Zou, Li, and Lv \cite{ZLL} obtain a necessary-and-sufficient classification for their regulated Beddington--DeAngelis model. Periodic coefficients and nonlinear jump perturbations are considered in \cite{JD,ZY}. These antecedents preclude a priority claim based solely on a classical response. A recent unpublished preprint by Bena\"im, Colombo, and Strickler (arXiv:2511.10113, 2025) treats degenerate Rosenzweig--MacArthur noise and quantitative convergence; its assumption of noise only in the prey differs from the two strictly positive noise intensities here. Our earlier prey-dependent Gause study \cite{Previous} concerns complementary extinction and invariance questions.

The contribution is a common stationary criterion with a direct proof, an explicit general-growth converse, and a precise distinction between a response's invasion limit and its interior dependence on predator density. Table~\ref{tab:scope} records the logical scope. No complete extinction classification or priority over all response-specific results is asserted. Section~\ref{sec:assumptions} gives the assumptions, Section~\ref{sec:results} states the results, Sections~\ref{sec:estimates}--\ref{sec:necessity} prove them, and Sections~\ref{sec:applications}--\ref{sec:numerics} give applications and numerical illustrations.

\begin{table}[t]
\centering\small
\caption{Scope of the stationary coexistence criteria}\label{tab:scope}
\begin{tabular}{p{.30\textwidth}p{.30\textwidth}p{.32\textwidth}}
\toprule Setting & Additional structure & Conclusion\\\midrule
General dissipative $F$ and admissible $\Phi$ & Positive prey establishment and predator invasion & Unique interior invariant measure\\[3pt]
General dissipative $F$ & \eqref{eq:no-fac-intro}, strictly increasing $\Phi(\cdot,0)$; prey establishment assumed & Interior stationarity requires positive predator invasion\\[3pt]
General dissipative $F$ with $F\leq F(0)$ & Same response conditions & Prey establishment and positive predator invasion are necessary and sufficient\\[3pt]
Logistic growth & Same response conditions & Explicit Gamma resident law and exact stationary criterion\\\bottomrule
\end{tabular}
\end{table}

\section{Model, assumptions, and boundary quantities}\label{sec:assumptions}
Let $M=[0,\infty)^2$, $E=(0,\infty)^2$, and $M_0=M\setminus E$. We work on a filtered probability space satisfying the usual conditions. All parameters in \eqref{eq:model} are strictly positive and the two standard Brownian motions are independent.

\begin{assumption}[Prey growth]\label{ass:F}
The function $F:[0,\infty)\to\R$ is $C^1$. There exist $C_F,a_F,d_F>0$ such that
\begin{equation}\label{eq:growth}
|F(x)|\leq C_F(1+x),\qquad xF(x)\leq a_F-d_Fx,\qquad x\geq0.
\end{equation}
\end{assumption}
\begin{assumption}[Trophic response]\label{ass:phi}
There is a nonnegative bounded locally Lipschitz function $\Psi:M\to[0,\infty)$ such that
\begin{equation}\label{eq:factor}
\Phi(x,y)=x\Psi(x,y),\qquad K_\Phi:=\sup_M\Psi<\infty.
\end{equation}
Moreover $\Phi(x,y)>0$ for $x>0$, $y\geq0$.
\end{assumption}
The factorization includes regularity at the prey-extinction axis. The inequality $0\leq\Phi\leq K_\Phi x$ alone is not equivalent to the full assumption: local Lipschitz regularity of the quotient at $x=0$ is also required. Neither monotonicity nor boundedness of $\Phi$ is assumed. The strict positivity implies $\Psi>0$ on $E$.

For a twice differentiable test function, the generator is
\begin{equation}\label{eq:generator}
\begin{split}
\Lcal f={}&[xF(x)-\Phi(x,y)y]f_x+y[\beta\Phi(x,y)-\delta]f_y\\
&+\tfrac12\sigma_x^2x^2f_{xx}+\tfrac12\sigma_y^2y^2f_{yy}.
\end{split}
\end{equation}
The logarithmic rates extend continuously to $M$:
\begin{equation}\label{eq:rates}
r_x=F(x)-\Psi(x,y)y-\sigma_x^2/2,\qquad
r_y=\beta\Phi(x,y)-\delta-\sigma_y^2/2.
\end{equation}
With
\begin{equation}\label{eq:W}
W=1+x+y/\beta,\quad d=\min(d_F,\delta),\quad
\sig=\max(\sigma_x^2,\sigma_y^2),
\end{equation}
we have $|r_x|+|r_y|\leq C W$. Section~\ref{sec:estimates} proves global well-posedness and the moment estimates needed below.

\begin{definition}
Stationary coexistence means that \eqref{eq:model} has an invariant probability measure $\pii$ with $\pii(E)=1$.
\end{definition}
Invariant boundary laws remain present even when stationary coexistence holds; uniqueness below concerns measures supported on $E$.

On the prey axis, the resident equation is
\begin{equation}\label{eq:resident}
\mathrm d\overline X_t=\overline X_t F(\overline X_t)\,\mathrm dt+
\sigma_x\overline X_t\,\mathrm dB_{x,t}.
\end{equation}
Define the prey establishment rate
\begin{equation}\label{eq:lambda-x}
\lambdax=F(0)-\sigma_x^2/2.
\end{equation}
If $\lambdax>0$, the unique resident invariant law on $(0,\infty)$ is $\mux$, with density
\begin{equation}\label{eq:resident-density}
p_X(x)=C x^{-2}\exp\left(\frac{2}{\sigma_x^2}\int_1^x\frac{F(u)}u\,\mathrm du\right),\qquad x>0.
\end{equation}
Its existence, integrability, and convergence from positive initial states are justified in Lemma~\ref{lem:resident}. Predator invasion is then well defined by
\begin{equation}\label{eq:lambda-y}
\lambday=\beta\int_0^\infty\Phi(x,0)\mux(\mathrm dx)-D_y,
\qquad D_y=\delta+\sigma_y^2/2.
\end{equation}
In particular, $\lambday$ is used only when the resident law exists.

\section{Main results}\label{sec:results}
\begin{theorem}[Sufficient invasion criterion]\label{thm:sufficient}
Under Assumptions~\ref{ass:F}--\ref{ass:phi}, if $\lambdax>0$ and $\lambday>0$, stationary coexistence holds. The interior invariant probability measure is unique. For every
\begin{equation}\label{eq:p-range}
1<p<1+\frac{2d}{\sig},
\end{equation}
it satisfies $\int_E W^p\dd\pii<\infty$.
\end{theorem}
The direct proof uses expected occupation measures and is given in Section~\ref{sec:existence}. Appendix~\ref{app:benaim} relates it to the more general persistence framework and establishes almost sure convergence of empirical occupation measures.

The converse uses the following additional assumption.
\begin{assumption}[Comparison structure]\label{ass:comparison}
For $x>0$, $y\geq0$,
\begin{equation}\label{eq:comparison}
\Phi(x,y)\leq\Phi(x,0),\qquad x\longmapsto\Phi(x,0)\text{ is strictly increasing}.
\end{equation}
\end{assumption}
\begin{theorem}[Necessity under prey establishment]\label{thm:conditional}
Under Assumptions~\ref{ass:F}--\ref{ass:comparison}, suppose $\lambdax>0$. Then stationary coexistence holds if and only if $\lambday>0$. Whenever it holds, the interior invariant law is unique.
\end{theorem}
\begin{theorem}[Exact criterion beyond logistic growth]\label{thm:exact}
Under Assumptions~\ref{ass:F}--\ref{ass:comparison}, assume in addition
\begin{equation}\label{eq:max-growth}
F(x)\leq F(0)\qquad(x\geq0).
\end{equation}
Stationary coexistence holds if and only if $\lambdax>0$ and, with the resulting resident law, $\lambday>0$. The interior invariant law is then unique.
\end{theorem}
Condition~\eqref{eq:max-growth} is satisfied by any nonincreasing prey per-capita growth satisfying Assumption~\ref{ass:F}, but monotonicity of $F$ is not needed. Theorem~\ref{thm:conditional} does not require \eqref{eq:max-growth}; it does require prey establishment as a separate hypothesis.

\begin{corollary}[Logistic growth]\label{cor:logistic}
Let $F(x)=r-cx$ with $r,c>0$, and suppose Assumptions~\ref{ass:phi} and \ref{ass:comparison} hold. Stationary coexistence is equivalent to
\begin{equation}\label{eq:logistic-criterion}
r-\sigma_x^2/2>0,\qquad
\beta\E_{\mux}\Phi(X,0)>\delta+\sigma_y^2/2,
\end{equation}
where $\mux=\operatorname{Gamma}(\alpha,\rho)$ in shape--rate notation,
\begin{equation}\label{eq:gamma}
\alpha=\frac{2r}{\sigma_x^2}-1,\qquad \rho=\frac{2c}{\sigma_x^2}.
\end{equation}
\end{corollary}

If $\lambdax>0$ but $\lambday=0$, Theorem~\ref{thm:conditional} excludes an interior invariant probability measure. This statement concerns stationary coexistence. It does not by itself imply almost sure convergence of $Y_t$ to zero, nor specify a critical extinction rate.

\section{Well-posedness, moments, and stationary identities}\label{sec:estimates}
\begin{proposition}\label{prop:moments}
Under Assumptions~\ref{ass:F}--\ref{ass:phi}, every initial state $z\in M$ gives a unique global strong solution of \eqref{eq:model}. The process is Feller on $M$, $E$ is invariant, and both axes are invariant. For each $p$ in \eqref{eq:p-range} there are $C_p,c_p>0$ such that
\begin{equation}\label{eq:drift-moment}
\Lcal W^p\leq C_p-c_pW^p,
\end{equation}
and
\begin{equation}\label{eq:moment}
\E_zW^p(Z_t)\leq e^{-c_pt}W^p(z)+\frac{C_p}{c_p}(1-e^{-c_pt}),\qquad t\geq0.
\end{equation}
Every invariant probability measure $\nu$ on $M$, whether interior or on the boundary, satisfies
\begin{equation}\label{eq:invariant-moment}
\int_M W^p\dd\nu\leq C_p/c_p.
\end{equation}
\end{proposition}
\begin{proof}
The coefficients are locally Lipschitz on $M$ and extend locally to Euclidean neighborhoods. Local SDE theory, as presented in \cite{Mao}, gives a unique local strong solution up to explosion. Multiplicative coordinates and pathwise uniqueness preserve the axes. For positive coordinates, the logarithmic equations are
\begin{equation}\label{eq:log-eqs}
\log X_t=\log x+\int_0^t r_x(Z_s)\,\mathrm ds+\sigma_xB_{x,t},\qquad
\log Y_t=\log y+\int_0^t r_y(Z_s)\,\mathrm ds+\sigma_yB_{y,t}.
\end{equation}
On every bounded region the integrands are bounded, so no coordinate reaches zero in finite time before explosion.

The balance \eqref{eq:balance-intro} gives $\Lcal W\leq a_F+d-dW$. Since
\[
\sigma_x^2x^2+\sigma_y^2y^2/\beta^2\leq\sig W^2,
\]
It\^o's formula yields
\[
\Lcal W^p\leq p(a_F+d)W^{p-1}
-p\left(d-\frac{p-1}{2}\sig\right)W^p.
\]
The bracket is positive. Absorbing $W^{p-1}$ into a small multiple of $W^p$ proves \eqref{eq:drift-moment}.

For a rigorous localization, set $U=W^p$ and $\tau_n=\inf\{t:W(Z_t)\geq n\}$, taking $n>W(z)$. The weaker inequality $\Lcal U\leq C_p$ gives
\[
\E_zU(Z_{t\wedge\tau_n})\leq U(z)+C_pt,
\qquad n^p\Prob_z(\tau_n\leq t)\leq U(z)+C_pt.
\]
Consequently explosion cannot occur. Applying It\^o to $e^{c_ps}U(Z_s)$ up to $t\wedge\tau_n$ gives
\[
\E_z[e^{c_p(t\wedge\tau_n)}U(Z_{t\wedge\tau_n})]
\leq U(z)+\frac{C_p}{c_p}(e^{c_pt}-1).
\]
Fatou's lemma as $n\to\infty$ proves \eqref{eq:moment}. This argument keeps the stopped time in the integrating factor; it does not replace an integral up to $t\wedge\tau_n$ by an unstopped differential inequality. Continuous dependence before localization, together with the exit estimate uniformly on compact initial sets, gives the Feller property.

It remains to prove \eqref{eq:invariant-moment} without assuming the desired moment. Choose smooth increasing concave $\vartheta_n$ on $[0,\infty)$ with $0\leq\vartheta_n'\leq1$, $\vartheta_n'=1$ on $[0,n]$, and $\vartheta_n'=0$ on $[2n,\infty)$. For instance take $\vartheta_n'(u)=\eta(u/n)$ for a smooth nonincreasing cutoff $\eta$. The function $\vartheta_n(U)$ is bounded; its generator and diffusion coefficient vanish outside a compact set and are bounded. Stationarity and It\^o therefore give
\[
0=\int\Lcal\vartheta_n(U)\dd\nu
\leq\int\vartheta_n'(U)(C_p-c_pU)\dd\nu.
\]
Thus $c_p\int\vartheta_n'(U)U\dd\nu\leq C_p$. Since $\vartheta_n'(U)\to1$ pointwise, Fatou proves \eqref{eq:invariant-moment}.
\end{proof}

\begin{lemma}[Stationary logarithmic identities]\label{lem:log}
Every invariant probability measure $\pi$ with $\pi(E)=1$ satisfies
\begin{equation}\label{eq:stationary-log}
\int r_x\dd\pi=\int r_y\dd\pi=0.
\end{equation}
No integrability of $\log x$ or $\log y$ is required.
\end{lemma}
\begin{proof}
Proposition~\ref{prop:moments} and $|r_i|\leq CW$ imply $r_i\in L^1(\pi)$. Start the process with law $\pi$. For any $\varepsilon>0$, stationarity gives
\[
\Prob\left(\frac{|\log X_t-\log X_0|}{t}>\varepsilon\right)
\leq 2\Prob(|\log X_0|>\varepsilon t/2)\longrightarrow0.
\]
The random variable $\log X_0$ is finite almost surely even if its expectation does not exist. The same argument applies to $Y$. Brownian motion divided by $t$ tends to zero almost surely. By the stationary ergodic theorem, $t^{-1}\int_0^t r_i(Z_s)\,\mathrm ds$ converges almost surely to a conditional expectation with respect to the invariant sigma-algebra. Equation~\eqref{eq:log-eqs} shows that this limit is also zero in probability; it must therefore equal zero almost surely. Taking expectations proves \eqref{eq:stationary-log}. This works without assuming ergodicity of $\pi$.
\end{proof}

\begin{lemma}[Resident law and boundary measures]\label{lem:resident}
Suppose $\lambdax>0$. Equation~\eqref{eq:resident} has the unique invariant law \eqref{eq:resident-density} on $(0,\infty)$, and its transition law from every positive initial state converges to $\mux$. Moreover
\begin{equation}\label{eq:resident-log}
\int F(x)\mux(\mathrm dx)=\sigma_x^2/2.
\end{equation}
Every invariant law on $M_0$ has the form
\begin{equation}\label{eq:boundary-mixture}
\nu_\theta=\theta\delta_{(0,0)}+(1-\theta)(\mux\otimes\delta_0),\qquad 0\leq\theta\leq1.
\end{equation}
\end{lemma}
\begin{proof}
The one-dimensional speed density is proportional to \eqref{eq:resident-density}. Since $F$ is $C^1$ at zero, its behavior there is $p_X(x)\asymp x^{2F(0)/\sigma_x^2-2}$, integrable precisely when $\lambdax>0$. For large $x$, \eqref{eq:growth} gives $F(x)/x\leq a_F/x^2-d_F/x$ and hence
\[
p_X(x)\leq Cx^{-2-2d_F/\sigma_x^2}.
\]
The speed measure is finite and has a finite first moment. The diffusion is conservative and regular on $(0,\infty)$, with positive diffusion coefficient. The one-dimensional positive-recurrence theorem gives uniqueness and convergence to its normalized speed measure; see the recurrence and ergodicity theory in \cite{Khasminskii}. Alternatively these follow in logarithmic coordinates from regular irreducible diffusion theory. Applying the argument of Lemma~\ref{lem:log} to the resident gives \eqref{eq:resident-log}; its required moment follows either from the density bound or the scalar version of Proposition~\ref{prop:moments}.

On the predator axis, $Y_t=y\exp(-D_yt+\sigma_y B_{y,t})\to0$ almost surely. There is no invariant probability supported on $(0,\infty)$ on that axis: invariance of $e^{-y}$ and dominated convergence would force its expectation to be one. The two open axes and the origin are disjoint invariant sets. Restricting a boundary invariant law to them leaves only the prey-axis invariant law and the origin, proving \eqref{eq:boundary-mixture}.
\end{proof}

\section{Direct construction of the coexistence law}\label{sec:existence}
\begin{proof}[Proof of Theorem~\ref{thm:sufficient}]
Fix $z\in E$, and let
\[
\overline\Pi_T^z=\frac1T\int_0^T P_t(z,\cdot)\,\mathrm dt
\]
be the expected occupation measure on $M$. For $p$ in \eqref{eq:p-range}, Proposition~\ref{prop:moments} gives $\sup_{T\geq1}\int W^p\dd\overline\Pi_T^z<\infty$. The sublevel sets of $W$ are compact in $M$, so these measures are tight. Choose a weak limit $\pi$ along $T_n\to\infty$. For $f\in C_b(M)$ and $s>0$,
\[
\left|\int P_sf\dd\overline\Pi_T^z-\int f\dd\overline\Pi_T^z\right|
\leq\frac{2s\|f\|_\infty}{T}.
\]
The Feller property permits passage to the weak limit. Thus $\pi$ is invariant.

For each finite $T$, \eqref{eq:log-eqs} can be averaged: $\E|\log X_T|$ and $\E|\log Y_T|$ are finite because the initial point is fixed and $\E\int_0^T|r_i(Z_s)|\,\mathrm ds<\infty$. Moreover $\log u\leq u$ and the uniform first moments give
\begin{equation}\label{eq:rates-limit}
\limsup_{T\to\infty}\int r_i\dd\overline\Pi_T^z
=\limsup_{T\to\infty}\frac{\E_z\log Z_{i,T}-\log z_i}{T}\leq0.
\end{equation}
Since $|r_i|\leq CW$ and $p>1$, the rates are uniformly integrable along these measures. Continuity and weak convergence imply
\begin{equation}\label{eq:nonpositive}
\int r_x\dd\pi\leq0,\qquad\int r_y\dd\pi\leq0.
\end{equation}

If $\pi(M_0)=1$, Lemma~\ref{lem:resident} gives $\pi=\nu_\theta$. Equation~\eqref{eq:resident-log} yields $\nu_\theta r_x=\theta\lambdax$. This is positive if $\theta>0$. If $\theta=0$, then $\nu_0r_y=\lambday>0$. Either possibility contradicts \eqref{eq:nonpositive}. Hence $\pi(E)>0$. Both $E$ and $M_0$ are invariant, with no transitions between them, so the normalized restriction $\pi(\cdot\cap E)/\pi(E)$ is an interior invariant probability. The moment conclusion for every admissible $p$ follows from \eqref{eq:invariant-moment}.

For uniqueness, apply $T(x,y)=(\log x,\log y)$. The transformed drift is locally Lipschitz on $\R^2$ and the covariance matrix is the positive definite constant matrix $\operatorname{diag}(\sigma_x^2,\sigma_y^2)$. Local elliptic regularity, localization, and nonexplosion give the strong Feller property. Every nonempty open set is reached with positive probability at every positive time: localize along a continuous path joining the initial point to that set, and use nondegenerate Brownian support, or the positive killed transition kernel on a bounded connected neighborhood of the path. Thus the transformed process is topologically irreducible. The strong Feller--irreducibility uniqueness criterion (see \cite{DPZ}) implies at most one invariant probability. The homeomorphism $T$ transfers uniqueness to $E$.
\end{proof}

\begin{remark}
The preceding direct construction first finds an invariant measure with positive interior mass and then normalizes its restriction. Uniqueness alone does not rule out a boundary component in an arbitrary expected occupation limit. Appendix~\ref{app:benaim} supplies the stronger pathwise persistence statement and convergence of the full empirical measures.
\end{remark}

\section{Stationary comparison and necessity}\label{sec:necessity}
\begin{lemma}[Pathwise prey comparison]\label{lem:comparison}
Let $(X,Y)$ solve \eqref{eq:model}, and let $\overline X$ solve \eqref{eq:resident} with the same $B_x$ and $\overline X_0=X_0>0$. Then
\begin{equation}\label{eq:path-order}
X_t\leq\overline X_t\quad\text{for all }t\geq0\quad\text{almost surely}.
\end{equation}
This statement uses only Assumptions~\ref{ass:F}--\ref{ass:phi}.
\end{lemma}
\begin{proof}
The Brownian terms cancel in the logarithmic difference $D_t=\log X_t-\log\overline X_t$. It is absolutely continuous, with
\[
\dot D_t=F(X_t)-F(\overline X_t)-\Psi(X_t,Y_t)Y_t
\quad\text{for almost every }t.
\]
Stop when $X$, $Y$, or $\overline X$ exceeds $R$. On $\{D_t>0\}$, local Lipschitz continuity gives $\dot D_t\leq L_R(X_t-\overline X_t)\leq L_RR D_t$. Therefore $(D_t^+)'\leq L_RR D_t^+$ almost everywhere before the stop. Since $D_0=0$, Gronwall implies $D_t^+=0$. Nonexplosion of the full and scalar systems allows $R\to\infty$, giving \eqref{eq:path-order}. In particular, no monotonicity of $F$ is needed.
\end{proof}

\begin{lemma}[Strict ordering of stationary prey laws]\label{lem:strict}
Suppose $\lambdax>0$ and an interior invariant law $\pi$ exists. Its prey marginal satisfies
\begin{equation}\label{eq:marginal-order}
\pi_X\st\mux,\qquad\pi_X\neq\mux.
\end{equation}
Consequently, if $g$ is strictly increasing and integrable for both laws, then
\begin{equation}\label{eq:strict-average}
\int g\dd\pi_X<\int g\dd\mux.
\end{equation}
\end{lemma}
\begin{proof}
Start the full system in $\pi$, independently of future Brownian increments, and couple $\overline X_0=X_0$. For bounded continuous increasing $h$, Lemma~\ref{lem:comparison} gives
\[
\int h\dd\pi_X=\E h(X_t)\leq\E h(\overline X_t).
\]
For a deterministic positive resident initial state, the right-hand transition law converges to $\mux$ by Lemma~\ref{lem:resident}. Integrating over the random $X_0$ and using bounded convergence proves the same conclusion for this initial law, and hence $\pi_X\st\mux$.

If the two laws coincided, Lemma~\ref{lem:log} and \eqref{eq:resident-log} would give
\[
0=\int r_x\dd\pi
=\int F\dd\mux-\sigma_x^2/2-\int\Psi(x,y)y\dd\pi
=-\int\Psi(x,y)y\dd\pi<0.
\]
The strict inequality holds since $\Psi y>0$ on $E$, and the integrals are finite by Proposition~\ref{prop:moments}. This contradiction proves distinctness. An ordered quantile coupling $(A,B)$ has $A\leq B$ almost surely and, because their laws differ, $\Prob(A<B)>0$. Strict monotonicity of $g$ yields $g(A)\leq g(B)$, strictly on an event of positive probability. Integrability gives \eqref{eq:strict-average}. This argument applies equally to bounded strictly increasing $g$.
\end{proof}

\begin{proof}[Proof of Theorem~\ref{thm:conditional}]
Sufficiency and uniqueness follow from Theorem~\ref{thm:sufficient}. For necessity, apply \eqref{eq:strict-average} with $g(x)=\Phi(x,0)$, which is integrable because $0\leq g(x)\leq K_\Phi x$. The predator identity and \eqref{eq:comparison} give
\[
D_y=\beta\int_E\Phi(x,y)\dd\pi
\leq\beta\int\Phi(x,0)\dd\pi_X
<\beta\int\Phi(x,0)\dd\mux.
\]
This is exactly $\lambday>0$.
\end{proof}
\begin{proof}[Proof of Theorem~\ref{thm:exact}]
It remains to infer prey establishment from an interior invariant law. By Lemma~\ref{lem:log},
\[
\lambdax=F(0)-\sigma_x^2/2
=\int_E[F(0)-F(x)+\Psi(x,y)y]\dd\pi>0.
\]
Here $F(0)-F(x)\geq0$ by \eqref{eq:max-growth}, and $\Psi(x,y)y>0$ on $E$. Theorem~\ref{thm:conditional} completes the proof.
\end{proof}
\begin{proof}[Proof of Corollary~\ref{cor:logistic}]
Logistic $F$ satisfies \eqref{eq:growth} and \eqref{eq:max-growth}. Substitution in \eqref{eq:resident-density} gives the Gamma density with parameters \eqref{eq:gamma}. Apply Theorem~\ref{thm:exact}.
\end{proof}

\section{Functional responses and nonlogistic growth}\label{sec:applications}
\subsection{Prey-dependent responses}
If $\Phi(x,y)=\varphi(x)$, condition~\eqref{eq:comparison} reduces to strict increase of $\varphi$. Under logistic prey growth,
\begin{equation}\label{eq:gamma-mean}
\E_{\mux}X=\frac{r-\sigma_x^2/2}{c}.
\end{equation}
The classical examples in Table~\ref{tab:responses} satisfy Assumption~\ref{ass:phi}. For Holling III, $\Psi(x)=ax/(1+bx^2)$ is bounded and locally Lipschitz. For Ivlev, $\Psi(x)=a(1-e^{-bx})/x$ extends smoothly at zero with value $ab$. These facts check the quotient regularity, not merely the upper bound on $\Phi$.

\begin{table}[t]
\centering\small
\caption{Admissible responses; all displayed parameters are positive}\label{tab:responses}
\begin{tabular}{lll}
\toprule Response & $\Phi(x,y)$ & $\Phi(x,0)$\\\midrule
Holling I & $qx$ & $qx$\\[3pt]
Holling II & $ax/(1+bx)$ & $ax/(1+bx)$\\[3pt]
Holling III & $ax^2/(1+bx^2)$ & $ax^2/(1+bx^2)$\\[3pt]
Ivlev & $a(1-e^{-bx})$ & $a(1-e^{-bx})$\\[3pt]
Beddington--DeAngelis & $ax/(1+bx+\eta y)$ & $ax/(1+bx)$\\[3pt]
Crowley--Martin type & $ax/[(1+bx)(1+\eta y)]$ & $ax/(1+bx)$\\\bottomrule
\end{tabular}
\end{table}

For $\varphi(x)=qx$, the logistic criterion is
\begin{equation}\label{eq:linear}
\lambday=\frac{\beta q}{c}(r-\sigma_x^2/2)-D_y>0.
\end{equation}
This is the known two-species food-chain threshold \cite{HNfood}. The stationary logarithmic identities recover its familiar means:
\begin{equation}\label{eq:linear-means}
\int X\dd\pii=\frac{D_y}{\beta q},\qquad
\int Y\dd\pii=\frac{\lambday}{\beta q^2}c.
\end{equation}
Indeed the prey balance is $q\E Y=r-\sigma_x^2/2-c\E X$, so the second expression follows from \eqref{eq:linear}. These formulas are recoveries, not new results.

For Holling II, the logistic feeding expectation is
\begin{equation}\label{eq:holling-integral}
m_\Phi=\frac{a\rho^\alpha}{\Gamma(\alpha)}
\int_0^\infty\frac{x^\alpha e^{-\rho x}}{1+bx}\,\mathrm dx,
\qquad\lambday=\beta m_\Phi-D_y.
\end{equation}
For Ivlev, the Gamma Laplace transform gives
\begin{equation}\label{eq:ivlev}
m_\Phi=a\left[1-\left(\frac{\rho}{\rho+b}\right)^\alpha\right].
\end{equation}
Holling III has the corresponding one-dimensional Gamma expectation. The necessary-and-sufficient criterion applies to each response without a response-specific recurrence construction.

\subsection{Predator interference and a common threshold}
Both predator-dependent responses in Table~\ref{tab:responses} satisfy \eqref{eq:comparison}. Their rare-predator response is the same as Holling II. Thus, for any fixed admissible $F$ with $F\leq F(0)$, these three response families have exactly the same stationary coexistence threshold. In particular the interference coefficient $\eta$ does not enter the criterion. This identifies an equivalence of thresholds, not equality of interior invariant laws.

The response condition is broader than monotone predator interference. Let $\varphi$ be any admissible strictly increasing prey-dependent response and take
\[
\Phi(x,y)=\varphi(x)\left(1-\varepsilon\frac{y}{1+y}\frac{1+\sin y}{2}\right),\qquad0<\varepsilon<1.
\]
The factor lies between $1-\varepsilon$ and one, equals one at $y=0$, and need not be monotone in $y$. Nevertheless all response assumptions hold, including \eqref{eq:comparison}. This example shows why comparison with the rare-predator limit is the relevant condition.

For any such family, let $m_\Phi=\int\Phi(x,0)\mux(\mathrm dx)>0$. The exact criterion can be written
\begin{equation}\label{eq:beta-crit}
\beta>\beta_c:=\frac{\delta+\sigma_y^2/2}{m_\Phi}.
\end{equation}
If $\beta m_\Phi>\delta$, the critical predator noise is
\begin{equation}\label{eq:sy-crit}
\sigma_{y,c}=\sqrt{2(\beta m_\Phi-\delta)},
\end{equation}
and coexistence requires $0<\sigma_y<\sigma_{y,c}$. Equality excludes an interior invariant law. If $\beta m_\Phi\leq\delta$, no positive predator-noise intensity permits coexistence in this class. These statements assume prey establishment and fix all quantities defining $m_\Phi$.

\subsection{A bounded prey-growth law}
Consider the dimensionless per-capita growth
\begin{equation}\label{eq:bounded-growth}
F(x)=r-k\frac{x}{1+x},\qquad k>r>0.
\end{equation}
It is bounded and decreasing, with $F(0)=r$ and limiting value $r-k<0$. It satisfies Assumption~\ref{ass:F}: for example choose $d_F=(k-r)/2$ and then a sufficiently large $a_F$. The establishment condition remains $r>\sigma_x^2/2$. Formula~\eqref{eq:resident-density} gives
\begin{equation}\label{eq:betaprime}
p_X(x)=\frac{x^{\alpha-1}(1+x)^{-\alpha-\zeta}}{B(\alpha,\zeta)},\quad
\alpha=\frac{2r}{\sigma_x^2}-1,\quad
\zeta=1+\frac{2(k-r)}{\sigma_x^2}.
\end{equation}
Thus the resident law is beta-prime with $\zeta>1$ and mean $\alpha/(\zeta-1)$. It has polynomial rather than exponential tails. Theorem~\ref{thm:exact} still gives an exact stationary criterion for every response in Table~\ref{tab:responses}. This example separates the comparison argument from the explicit Gamma calculations available only for logistic growth.

\section{Reproducible numerical illustrations}\label{sec:numerics}
The computations illustrate two statements: the location of the invasion threshold and the distinction between a common threshold and different interior distributions. They are not a substitute for the proofs. All parameters are dimensionless; no observational data are fitted.

For the logistic model and Beddington--DeAngelis response, take
\begin{equation}\label{eq:num-parameters}
r=c=1,\quad a=1.5,\quad b=1,\quad\beta=1.5,\quad
\delta=0.5,\quad\sigma_x=0.6,\quad\sigma_y=0.5.
\end{equation}
Here $\alpha=41/9$ and $\rho=50/9$. Adaptive quadrature of \eqref{eq:holling-integral} gives
\[
m_\Phi\approx\NumericalFeeding,\qquad
\lambday\approx\NumericalLambda,\qquad
\sigma_{y,c}\approx\NumericalCriticalNoise.
\]
The quadrature was repeated with tighter absolute and relative tolerances; the feeding expectations differed by less than $10^{-8}$. Figure~\ref{fig:threshold} evaluates the same expectation as $\sigma_x$ varies and plots the zero contour of $\lambday$. The figure is restricted to prey-establishing noise intensities. Negative values exclude stationary coexistence by Corollary~\ref{cor:logistic}; no pathwise extinction claim is inferred from the plot.

\begin{figure}[t]
\centering\includegraphics[width=.88\textwidth]{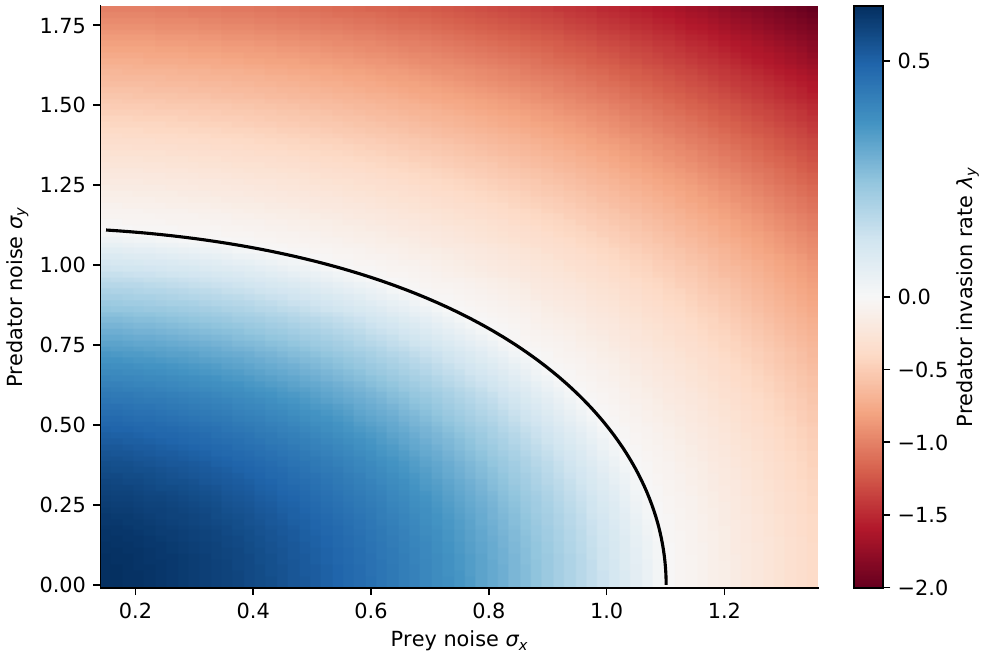}
\caption{Predator invasion rate for Holling II and the corresponding Beddington--DeAngelis and Crowley--Martin families. The solid contour is $\lambday=0$. Parameters are \eqref{eq:num-parameters}, except for the displayed noise intensities. The heatmap uses deterministic Gamma quadrature, not an estimate from trajectories}\label{fig:threshold}
\end{figure}

Figure~\ref{fig:resident} compares resident laws for logistic growth and \eqref{eq:bounded-growth} with $k=1.8$, and displays the feeding curves at positive predator density. The latter depend on $\eta$, although their values at $y=0$ coincide.
\begin{figure}[t]
\centering\includegraphics[width=\textwidth]{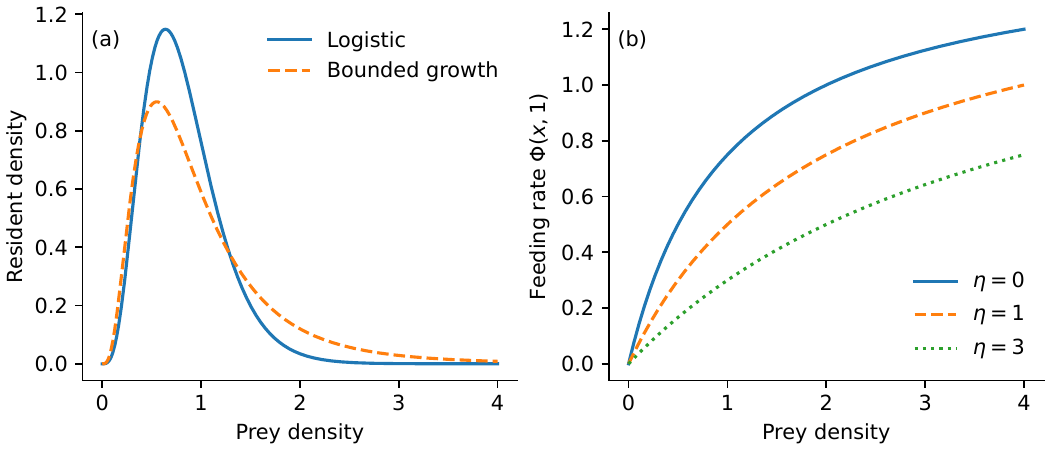}
\caption{(a) Explicit resident prey densities for logistic and bounded growth at $r=1$, $\sigma_x=0.6$, with $c=1$ or $k=1.8$, respectively. (b) Beddington--DeAngelis feeding rates at $y=1$ for three interference intensities; all three have the same rare-predator response}\label{fig:resident}
\end{figure}

For the full model, simulations use Euler--Maruyama in logarithmic coordinates, following the numerical SDE construction described in \cite{Higham}. With $U_n=\log X_n$, $V_n=\log Y_n$,
\begin{equation}\label{eq:scheme}
\begin{aligned}
U_{n+1}&=U_n+h\,r_x(e^{U_n},e^{V_n})+\sigma_x\sqrt h\,\xi_{x,n},\\
V_{n+1}&=V_n+h\,r_y(e^{U_n},e^{V_n})+\sigma_y\sqrt h\,\xi_{y,n},
\end{aligned}
\end{equation}
where the innovations are independent standard normal variables. Exponentiation preserves positivity; there is no reflection, clipping, or artificial extinction floor. We use 32 independent Brownian replicates, $T=1200$, burn-in 200, and $(X_0,Y_0)=(0.8,0.4)$. Fine increments at $h=0.01$ are aggregated to couple a second run at $h=0.02$. Common increments are also used across interference values $\eta=0,1,3$. The random seed is 20261020.

Table~\ref{tab:numerics} gives mean time averages, their standard errors across the independent replicates, and the time-step comparison. These standard errors measure between-replicate Monte Carlo uncertainty, not discretization or finite-horizon bias. Figure~\ref{fig:occupations} uses observations every 0.5 units after burn-in at the finer step. The empirical distributions differ across interference levels despite the shared positive invasion rate. They are finite-horizon numerical approximations of stationary occupations; the discretization is not claimed to preserve the invariant law exactly.

\begin{table}[t]
\centering\small
\caption{Time averages from the full logistic model. SE is the standard error across 32 replicate time averages at $h=0.01$; $\Delta$ is the mean paired difference (fine minus coarse)}\label{tab:numerics}
\begin{tabular}{rrrrrr}
\toprule $\eta$ & Coordinate & Mean, $h=.01$ & SE & Mean, $h=.02$ & $\Delta$\\\midrule
\NumericalRows
\bottomrule
\end{tabular}
\end{table}
\begin{figure}[t]
\centering\includegraphics[width=\textwidth]{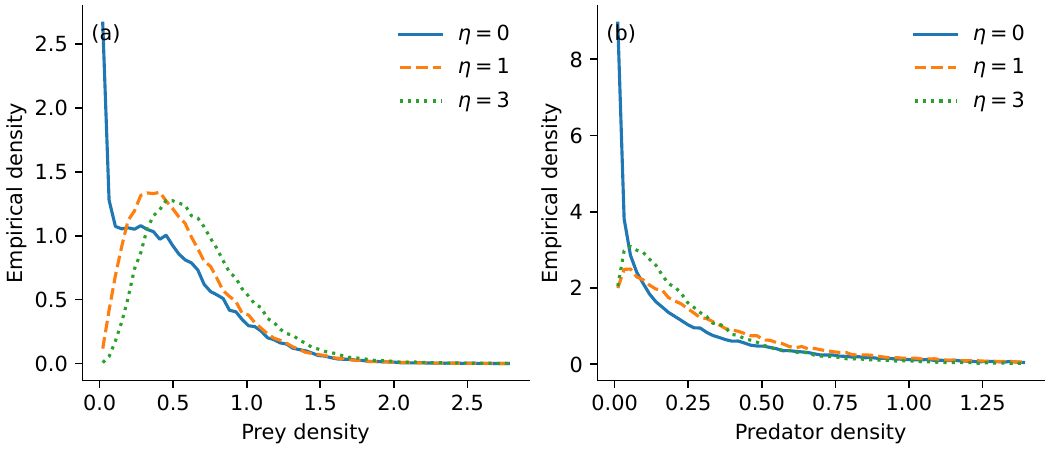}
\caption{Empirical marginal occupation densities for (a) prey and (b) predators, with the three interference intensities in Table~\ref{tab:numerics}. Line styles distinguish the models. The calculation pools 32 trajectories after burn-in; samples within a trajectory are temporally dependent}\label{fig:occupations}
\end{figure}

As independent checks, one million Gamma draws gave a feeding mean of $0.641245$ with Monte Carlo standard error $0.000169$, consistent with the quadrature value. The beta-prime density in \eqref{eq:betaprime} was checked for unit mass and its stated mean by numerical integration. Across the examples in Table~\ref{tab:numerics}, the largest absolute change of a population mean on halving the step was less than $4\times10^{-4}$. The sampled logarithmic drift means had absolute values below $6\times10^{-4}$, within their between-replicate standard errors. These diagnostics do not constitute an invariant-measure discretization error bound.

Online Resource~1 contains the Python reproduction and validation scripts, parameters, software versions, CSV summaries, and sampled occupations. The figures were produced using NumPy, SciPy, and Matplotlib from the stated equations.

\section{Discussion}
The exact criterion identifies a biological consequence of prey limitation and the absence of positive predator facilitation. A resident prey environment is at least as favorable to a rare predator as the prey distribution maintained under coexistence, and strictly more favorable on average when feeding increases strictly with prey density. Thus a predator maintained in an interior stationary regime must be able to invade the resident. The argument remains valid for saturating feeding responses and for nonlogistic growth with its largest per-capita value at zero.

Predator interference has two distinct roles. Its absence from $\Phi(x,0)$ makes it invisible to the invasion threshold. Its effect on $\Phi(x,y)$ for $y>0$ changes the feedbacks operating after establishment and can alter the stationary population distribution. The numerical examples illustrate this distinction; they do not establish a general monotonicity theorem for stationary means as interference varies.

The proof also separates two technical issues that should not be conflated. A positive moment estimate from deterministic initial data does not immediately establish the same moment for every pre-existing invariant law; the concave-cutoff argument provides that step. Similarly, positive moments do not control logarithmic singularities near zero. Stationary logarithmic identities follow instead from convergence in probability of stationary log increments divided by time and the ergodic theorem for the integrable drifts.

The general sufficient result fits existing stochastic persistence theory. The weighted trophic balance gives the verification needed for this model class; it is not a replacement for the invasion principle. In the logistic case a particular global condition in \cite{HNkol} fails, as shown in Appendix~\ref{app:HN}, but a more flexible abstract theorem still applies. Failure of one sufficient hypothesis is therefore not evidence that a model lies outside all previous persistence theory.

There are several limits to the conclusions. Cooperative feeding can violate \eqref{eq:comparison}, in which case the general sufficient criterion still holds but the necessity proof does not. Without $F\leq F(0)$, Theorem~\ref{thm:conditional} assumes rather than deduces prey establishment. Degenerate noise, correlated noises, and a full classification of critical paths are outside the present statements. In particular, exclusion of an interior invariant law at $\lambday=0$ is weaker than a complete critical extinction theorem. Quantitative convergence rates and statistically validated comparison with ecological observations are additional questions.

\appendix
\section{Verification of the abstract persistence framework}\label{app:benaim}
This appendix verifies that Theorem~\ref{thm:sufficient} is an application of Theorem~4.4 and Corollary~4.5 in Bena\"im's \emph{Stochastic Persistence}, arXiv:1806.08450v3, specified in the introduction. We use its abstract hypotheses, rather than its stronger response-specific ecological sufficient conditions.

The state space $M$ is locally compact and Polish, $M_0$ is closed and invariant, and the process is conservative and Feller by Proposition~\ref{prop:moments}. These give its standing hypotheses. Choose
\[
1<q<p<\min\{2,1+2d/\sig\},\qquad U=W^q.
\]
Then $\Lcal U\leq C_q-c_qU$, and $U$, $\widetilde U=c_qU$ are proper maps on $M$. To verify the extended-generator strong law, write $M_t^U=U(Z_t)-U(Z_0)-\int_0^t\Lcal U(Z_s)\,\mathrm ds$.

The drift of $W$ is bounded above by a constant $A$, and its diffusion vector has norm at most $\sigma_*W$. For $n\leq t\leq n+1$, It\^o's formula gives
\[
\log W(Z_t)-\log W(Z_n)\leq A+N_t-N_n,\qquad
\langle N\rangle_{n+1}-\langle N\rangle_n\leq\sig.
\]
The conditional exponential martingale inequality therefore implies
\[
\E\left[\sup_{n\leq t\leq n+1}W(Z_t)^p\mid\mathcal F_n\right]\leq C W(Z_n)^p.
\]
Together with \eqref{eq:moment}, the unconditional expectations are uniformly bounded for every fixed initial state. For $s=p/q\in(1,2)$, Burkholder--Davis--Gundy gives
\[
\E\sup_{n\leq t\leq n+1}|M_t^U-M_n^U|^s
\leq C\E\left(\int_n^{n+1}W(Z_t)^{2q}\,\mathrm dt\right)^{s/2}
\leq C_z.
\]
For the integer increments $D_n=M_n^U-M_{n-1}^U$, the sum $\sum_n\E|D_n/n|^s$ is finite. The martingale convergence criterion in $L^s$ and Kronecker's lemma yield $M_n^U/n\to0$ almost surely. Markov's inequality and Borel--Cantelli applied to the interval suprema give the continuous-time conclusion. On finite horizons all moments of $\sup W$ are finite by the same exponential estimate for $\log W$; thus the localizations used to define these true martingales can be removed. This verifies the control-at-infinity strong law.

For the boundary function take $a,b>0$ with $a+b=1$ and $a\lambdax-bD_y>0$, and set
\[
V=\log W-a\log x-b\log(y/\beta),\qquad H=\Lcal\log W-ar_x-br_y.
\]
Weighted arithmetic--geometric mean implies $V\geq0$ on $E$. The function $H$ extends continuously to $M$ and $|H|\leq CW$. Indeed
\[
\Lcal\log W=\frac{xF(x)-\delta y/\beta}{W}
-\frac{\sigma_x^2x^2+\sigma_y^2y^2/\beta^2}{2W^2}.
\]
The diffusion coefficients of $V$ are $\sigma_x(x/W-a)$ and $\sigma_y(y/(\beta W)-b)$, uniformly bounded. Its martingale part divided by time converges almost surely to zero. Also $\widetilde U/(1+|H|)\geq cW^{q-1}$, hence that map is proper. These establish the boundary-function hypotheses.

The ergodic boundary laws are $\mu_0=\delta_{(0,0)}$ and $\mu_x=\mux\otimes\delta_0$. In both, $\int\Lcal\log W\dd\mu=0$: $\log W$ is integrable, its diffusion is bounded, and its drift is integrable by the moment estimate. Thus
\[
\mu_0H=-a\lambdax+bD_y<0,\qquad\mu_xH=-b\lambday<0.
\]
The abstract persistence exponent is positive. The cited theorem and corollary, together with the uniqueness already proved, give for every $z\in E$
\begin{equation}\label{eq:empirical-convergence}
\frac1T\int_0^T\delta_{Z_t}\,\mathrm dt\ \Longrightarrow\ \pii
\quad\text{almost surely as }T\to\infty.
\end{equation}
They also give persistence in occupation: for every $\varepsilon>0$ a compact $K_\varepsilon\subset E$ exists for which the almost sure lower limit of the fraction of time spent in $K_\varepsilon$ is at least $1-\varepsilon$. No convergence rate follows from this verification.

\section{A particular dissipativity condition that fails}\label{app:HN}
For completeness, we show why Hypothesis~1.1(3) in \cite{HNkol} is not directly available for logistic growth in \eqref{eq:model}. Use the notation $f_1=r-cx-\Psi y$, $f_2=\beta\Phi-\delta$. That hypothesis requires some $u,v,\gamma>0$ for which the upper limit at infinity of
\[
Q=\frac{uxf_1+vyf_2}{1+ux+vy}
-\frac{u^2\sigma_x^2x^2+v^2\sigma_y^2y^2}{2(1+ux+vy)^2}
+\gamma(1+|f_1|+|f_2|+\sigma_x^2+\sigma_y^2)
\]
is negative. Fix arbitrary positive $u,v,\gamma$, choose $x_0>r/c$, put $A=cx_0-r$, and write $D=1+ux_0+vy$. Then
\[
ux_0f_1+vyf_2=-ux_0A+(v\beta-u)\Phi(x_0,y)y-v\delta y.
\]
Since $|f_1(x_0,y)|=A+\Phi(x_0,y)y/x_0$ and the diffusion correction is at least $-\sig/2$,
\[
Q(x_0,y)\geq-\frac{ux_0A}{D}-\delta-\frac{\sig}{2}
+\gamma A+\Phi(x_0,y)\left(\frac{\gamma y}{x_0}-\frac{|v\beta-u|}{v}\right).
\]
For large $y$ the last term is nonnegative. Choosing $x_0$ large enough gives $\liminf_{y\to\infty}Q(x_0,y)>0$, contradicting the required condition. The proof allows predator-dependent responses. It establishes failure of this particular hypothesis for logistic growth, not failure of all persistence frameworks; Appendix~\ref{app:benaim} verifies another one.

\section*{Statements and Declarations}
\input{submission/declarations.tex}

\end{document}

%% file: numerics/results.tex
\newcommand{\NumericalFeeding}{0.641401}
\newcommand{\NumericalLambda}{0.337102}
\newcommand{\NumericalCriticalNoise}{0.961355}

%% file: numerics/table.tex
\newcommand{\NumericalRows}{%
0 & $X$ & 0.4713 & 0.0024 & 0.4716 & -0.00034 \\
0 & $Y$ & 0.2991 & 0.0029 & 0.2987 & +0.00039 \\
1 & $X$ & 0.5598 & 0.0023 & 0.5599 & -0.00011 \\
1 & $Y$ & 0.3793 & 0.0047 & 0.3792 & +0.00009 \\
3 & $X$ & 0.6821 & 0.0026 & 0.6821 & -0.00004 \\
3 & $Y$ & 0.2690 & 0.0036 & 0.2690 & +0.00003 \\
}

%% file: author_details.tex
\newcommand{\AuthorOneDetails}{}
\newcommand{\AuthorTwoDetails}{}
\newcommand{\AuthorThreeDetails}{}

%% file: submission/declarations.tex
\paragraph{Data and code availability.}
No observational data were used. The numerical examples are generated from the equations and parameter values stated in the article. Online Resource~1 contains the complete reproduction script, random seed, software versions, numerical summaries, and sampled occupations.

\paragraph{Use of generative AI.}
Claude (Anthropic) was used only to check and correct the English language of the manuscript and to assist with typing parts of the manuscript in \LaTeX. All mathematical results, proofs, numerical code, and interpretations are the work of the authors, who take full responsibility for the content of this article.

%% file: main.bbl
\begin{thebibliography}{99}
\bibitem{Holling} Holling, C.S.: Some characteristics of simple types of predation and parasitism. Can. Entomol. \textbf{91}, 385--398 (1959). \url{https://doi.org/10.4039/Ent91385-7}
\bibitem{Beddington} Beddington, J.R.: Mutual interference between parasites or predators and its effect on searching efficiency. J. Anim. Ecol. \textbf{44}, 331--340 (1975). \url{https://doi.org/10.2307/3866}
\bibitem{DeAngelis} DeAngelis, D.L., Goldstein, R.A., O'Neill, R.V.: A model for trophic interaction. Ecology \textbf{56}, 881--892 (1975). \url{https://doi.org/10.2307/1936298}
\bibitem{Crowley} Crowley, P.H., Martin, E.K.: Functional responses and interference within and between year classes of a dragonfly population. J. N. Am. Benthol. Soc. \textbf{8}, 211--221 (1989). \url{https://doi.org/10.2307/1467324}
\bibitem{Kot} Kot, M.: Elements of Mathematical Ecology. Cambridge University Press, Cambridge (2001). \url{https://doi.org/10.1017/CBO9780511608520}
\bibitem{SBA} Schreiber, S.J., Bena\"im, M., Atchad\'e, K.A.S.: Persistence in fluctuating environments. J. Math. Biol. \textbf{62}, 655--683 (2011). \url{https://doi.org/10.1007/s00285-010-0349-5}
\bibitem{HNkol} Hening, A., Nguyen, D.H.: Coexistence and extinction for stochastic Kolmogorov systems. Ann. Appl. Probab. \textbf{28}, 1893--1942 (2018). \url{https://doi.org/10.1214/17-AAP1347}
\bibitem{HNC} Hening, A., Nguyen, D.H., Chesson, P.: A general theory of coexistence and extinction for stochastic ecological communities. J. Math. Biol. \textbf{82}, 56 (2021). \url{https://doi.org/10.1007/s00285-021-01606-1}
\bibitem{HL} Hening, A., Li, Y.: Stationary distributions of persistent ecological systems. J. Math. Biol. \textbf{82}, 64 (2021). \url{https://doi.org/10.1007/s00285-021-01613-2}
\bibitem{HNfood} Hening, A., Nguyen, D.H.: Stochastic Lotka--Volterra food chains. J. Math. Biol. \textbf{77}, 135--163 (2018). \url{https://doi.org/10.1007/s00285-017-1192-8}
\bibitem{NS} Nguyen, D.H., Strickler, E.: A method to deal with the critical case in stochastic population dynamics. SIAM J. Appl. Math. \textbf{80}, 1567--1589 (2020). \url{https://doi.org/10.1137/20M131134X}
\bibitem{DNY} Du, N.H., Nguyen, D.H., Yin, G.: Conditions for permanence and ergodicity of certain stochastic predator--prey models. J. Appl. Probab. \textbf{53}, 187--202 (2016). \url{https://doi.org/10.1017/jpr.2015.18}
\bibitem{LG} Li, S., Guo, S.: Permanence of a stochastic prey--predator model with a general functional response. Math. Comput. Simul. \textbf{187}, 308--336 (2021). \url{https://doi.org/10.1016/j.matcom.2021.02.025}
\bibitem{ZLL} Zou, X., Li, Q., Lv, J.: Stochastic bifurcations, a necessary and sufficient condition for a stochastic Beddington--DeAngelis predator--prey model. Appl. Math. Lett. \textbf{117}, 107069 (2021). \url{https://doi.org/10.1016/j.aml.2021.107069}
\bibitem{JD} Ji, W., Deng, M.: Coexistence and extinction of a periodic stochastic predator--prey model with general functional response. Adv. Difference Equ. \textbf{2020}, 436 (2020). \url{https://doi.org/10.1186/s13662-020-02898-1}
\bibitem{ZY} Zhang, X., Yang, Q.: Dynamical behavior of a stochastic predator-prey model with general functional response and nonlinear jump-diffusion. Discrete Contin. Dyn. Syst. B \textbf{27}, 3155--3175 (2022). \url{https://doi.org/10.3934/dcdsb.2021177}
\bibitem{Previous} S\'anchez, A., Valencia, L.A., Ramirez Osorio, J.M.: The stochastic Gause predator--prey model: Noise-induced extinctions and invariance. Math. Biosci. Eng. \textbf{22}, 1999--2019 (2025). \url{https://doi.org/10.3934/mbe.2025073}
\bibitem{Mao} Mao, X.: Stochastic Differential Equations and Applications, 2nd edn. Horwood Publishing, Chichester (2007). \url{https://doi.org/10.1533/9780857099402}
\bibitem{Khasminskii} Khasminskii, R.: Stochastic Stability of Differential Equations, 2nd edn. Stochastic Modelling and Applied Probability, vol. 66. Springer, Berlin (2012). \url{https://doi.org/10.1007/978-3-642-23280-0}
\bibitem{DPZ} Da Prato, G., Zabczyk, J.: Ergodicity for Infinite Dimensional Systems. London Mathematical Society Lecture Note Series, vol. 229. Cambridge University Press, Cambridge (1996). \url{https://doi.org/10.1017/CBO9780511662829}
\bibitem{Higham} Higham, D.J.: An algorithmic introduction to numerical simulation of stochastic differential equations. SIAM Rev. \textbf{43}, 525--546 (2001). \url{https://doi.org/10.1137/S0036144500378302}
\end{thebibliography}
